\documentclass{amsart}

\usepackage[ruled,vlined, algo2e]{algorithm2e}
\usepackage[english, activeacute]{babel}
\usepackage[latin1]{inputenc}
\usepackage{float}
\usepackage{hyperref}

\usepackage[all]{xy}
\usepackage{amsmath}
\usepackage{amssymb}
\usepackage[left=3cm, right=3cm, top=3cm, bottom=3cm]{geometry}

\newcommand{\SG}{\mathrm{SG}}
\newcommand{\e}{\mathrm{e}}
\newcommand{\g}{\mathrm{g}}

\newcommand{\F}{\mathrm{F}}
\newcommand{\G}{\mathrm{G}}
\newcommand{\m}{\mathrm{m}}
\newcommand{\Ap}{\mathrm{Ap}}
\newcommand{\Z}{\mathbb{Z}}
\newcommand{\N}{\mathbb{N}}
\newcommand{\R}{\mathbb{R}}

\newcommand{{\km}}{\rm k}
\newcommand{\dsum}{\displaystyle\sum}

\newtheorem{theo}{Theorem}
\newtheorem{defi}[theo]{Definition}
\newtheorem{ex}[theo]{Example}
\newtheorem{lem}[theo]{Lemma}
\newtheorem{cor}[theo]{Corollary}
\newtheorem{prop}[theo]{Proposition}
\newtheorem{rem}[theo]{Remark}

\title[Integer Programming and $m$-irreducibility of numerical semigroups]{Integer Programming and $m$-irreducibility of numerical semigroups}

\author[V. Blanco]{V\'ictor Blanco}
\address{Departamento de \'Algebra, Universidad de Granada}
\email{vblanco@ugr.es}

\author[J. Puerto]{Justo Puerto}
\address{Departamento de Estad\'istica e Investigaci\'on Operativa, Universidad de Sevilla}
\email{puerto@us.es}
\date{\today}

\keywords{
integer programming, numerical semigroups, irreducibility, multiplicity.}

\subjclass[2010]{90C10, 20M14, 11D75}

\begin{document}

\begin{abstract}
This paper addresses the problem of decomposing a numerical semigroup into $m$-irreducible
numerical semigroups. The problem originally stated in algebraic terms is translated, introducing the so called Kunz-coordinates, to resolve a series of several discrete optimization problems. First, we prove that finding a minimal $m$-irreducible decomposition is equivalent to solve a multiobjective linear integer problem. Then, we restate that problem as the problem of finding all the optimal solutions of a finite number of single objective integer linear problems plus  a set covering problem. Finally, we prove that there is a suitable transformation that reduces the original  problem to find an optimal solution of a compact integer linear problem. This result ensures a polynomial time algorithm for each given
multiplicity $m$. We have implemented the different algorithms and have performed some computational experiments to show the efficiency of our methodology.
\end{abstract}

\maketitle

\section{Introduction}

The rich literature of discrete mathematics contains an important number of
references,
of theoretical results, that have helped in solving or advancing in the
resolution of many different discrete optimization problems. Nowadays, it is  considered a standard  to cite the connections between graph theory, commutative algebra and optimization.
For instance, it is a topic to mention the connections between the results by K\"onig and Egervary in graph theory and the algorithms by Kuhn for the assignment problem or by Edmonds for the maximum matching problem (\cite{edmonds,egervary,konig,kuhn,lovasz}). In the same vein, there are well-known papers that apply algebraic tools to solve single objective integer linear problems (see \cite{conti-traverso91},
\cite{deloera04}, \cite{lasserre}) or multiobjective integer linear problems (see \cite{blanco-puerto09,blanco-puerto09a,blanco-puerto09:poly, deloera09}).
Moreover, more recently we can find a rich  body of literature (see \cite{onn2010} and the references therein) that addresses general discrete optimization problems (non necessarily linear)
 using tools borrowed from pure and applied Algebra (Graver bases and the like).

On the other hand, although less known,  there have been also some applications of integer programming to solve problems of
commutative algebra (see \cite{ jj07,counting,blanco10}).
The goal of this paper is to analyze and solve another problem arising in commutative algebra using tools from integer programming.

A numerical semigroup is a subset $S$ of $\Z_+$ (here $\Z_+$ denotes the set of non-negative integers) closed under
addition, containing zero and such that $\Z_+\backslash S$ is finite.
 Numerical semigroups were first considered while studying the
set of nonnegative solutions of Diophantine equations and their study is closely related to the analysis of monomial curves (see \cite{delorme}). By these reasons, the theory of numerical semigroups has attracted a number of researchers from the algebraic community. This fact has motivated that some of the terminology used in Algebraic Geometry has been exported to this field. For instance, the multiplicity, the genus, or the embedding dimension of a numerical semigroup.   Further details about the theory of numerical semigroups can be found in the recent monograph by Rosales and Garc\'ia-S\'anchez \cite{Ro-GS09}.

In recent years, the problem of decomposing numerical semigroups into irreducible ones has attracted the interest of the research community (see \cite{branco-nuno07, rosales02, rosales03, rosales03b, rosales04}). Recall that  a numerical semigroup is irreducible if it cannot be expressed as
an intersection of two numerical semigroups containing it properly.
Furthermore, more recently a different notion of irreducibility, the $m$-irreducibility (\cite{blanco-rosales10}) has appeared and has started to be analyzed.
A numerical semigroup with multiplicity $m$ is said $m$-irreducible if it cannot be expressed  as an intersection of two numerical semigroups with multiplicity
 $m$ and containing it properly. (Recall that the multiplicity of a numerical semigroup is the smallest non zero element belonging to it.) The question of
 existence of $m$-irreducible decompositions has been proved in \cite{blanco-rosales10}. Nevertheless, it is still missing a methodology, different to the
 almost pure brute force enumeration, to find $m$-irreducible decompositions of minimal size.

In this paper, we give a methodology to obtain such a minimal decomposition into $m$-irreducible numerical semigroups by
using tools borrowed from discrete optimization. To this end, we identify one-to-one  numerical semigroups with  the integer
 vectors inside a rational polytope (see \cite{rosales02-london}).  For the sake of this identification, we introduce the notion of Kunz-coordinates vector
to translate the considered problem in the problem of finding some integer optimal solutions, with respect to appropriate objective
 functions,   in the Kunz polytope. Then, the problem of enumerating the minimal $m$-irreducible numerical semigroups involved in
 the decomposition is formulated as a multiobjective integer program. We state that solving this problem is equivalent
 to enumerate the entire sets of optimal solutions of a finite set of single-objective integer problems. The number of
 integer problems to be solved is bounded above by $m-1$, where $m$ is the multiplicity of the semigroup to be decomposed.
 Finally, we solve a set covering problem to ensure that the decomposition has the smallest number of elements.
Although this approach is exact its complexity is rather high and
in general one cannot prove that it is polynomial for any given multiplicity
$m$. This comes from the fact that there are nowadays relatively few exact methods to solve general
multiobjective integer and linear problems (see \cite{ehrgott02}) and it is known that the complexity of solving in general this type of problems is $\#$P-hard.
To overcome this difficulty, we introduce a different machinery that identify a minimal decomposition by 
solving a compact  linear integer program. This approach ensures that the problem of finding a minimal
$m$-irreducible decomposition is  polynomially  solvable.

In Section \ref{sec:1} we recall the main definitions and results needed for this paper to be selfcontained. Section \ref{sec:2}
 is devoted to translate the problem of finding numerical semigroups of a given multiplicity into the problem of detecting integer
points inside a rational polytope, introducing the notion of Kunz-coordinates vector. We write in Section \ref{sec:3} the conditions,
in terms of the Kunz-coordinates vector fo a numerical semigroup to be a $m$-irreducible oversemigroup. Section \ref{sec:4} is devoted
to formulate the problem of decomposing and minimally decomposing into $m$-irreducible numerical semigroups as a mathematical programming problem.
We give an exact and a heuristic approach for computing such a minimal decomposition based on solving some integer programming problems.
In Section \ref{sec:5} we present a compact model to compute, by solving only one integer programming problem, a minimal decomposition of
 a numerical semigroup into $m$-irreducible numerical semigroups. There, we also prove that this problem is  polynomially solvable. 
 Finally, in Section \ref{sec:6} we show some computational tests performed to check the efficiency of the presented algorithms with respect to the current implementation in {\sf GAP} \cite{numericalsgps}.

\section{Preliminaries}
\label{sec:1}
For the sake of readability, in this section we recall the main results about numerical semigroups needed so that the paper is selfcontained.

 Let $S$ be a numerical semigroup. We say that $\{n_1, \ldots, n_p\}$ is a system of
generators of $S$ if $S=\{\dsum_{i=1}^p n_i x_i: x_i \in \Z_+, i=1, \ldots, p\}$. We denote $S=\langle n_1, \ldots, n_p\rangle$ if $\{n_1, \ldots, n_p\}$
is a system of generators of $S$.

The least positive integer belonging to $S$ is denoted by $\m(S)$, and is called the \textit{multiplicity of} $S$ ($\m(S) = \min(S\setminus
\{0\})$).

Two important notions of irreducibility are extensively used through this paper. They are the following:
\begin{defi}[Irreducibility and m-irreducibility] $ $
\label{def:1}
\begin{itemize}
 \item A numerical semigroup is \emph{irreducible} if it cannot be expressed as an intersection of two numerical semigroups containing it properly.
\item A numerical semigroup with multiplicity $m$ is $m$-\emph{irreducible} if it cannot be expressed as an intersection of two numerical semigroups with multiplicity $m$ containing it properly.
\end{itemize}
\end{defi}
In \cite{blanco-rosales10} the authors analyze and characterize the set of $m$-irreducible numerical semigroups. Note that, in particular, any irreducible numerical semigroup is $m$-irreducible, while the converse is not true. One of the results in that paper is the key for the analysis done through this paper and it is stated as follows.
\begin{prop}[\cite{blanco-rosales10}]
\label{prop:2}
Let $S$ be a numerical semigroup with multiplicity $m$. Then, there exist $S_1, \ldots, S_k$ $m$-irreducible numerical semigroups such that $S=S_1 \cap \cdots \cap S_k$.
\end{prop}

From the above result one may think of obtaining the minimal number of elements involved in the above intersection of $m$-irreducible numerical semigroup. Formally, we describe what we understand by decomposing and minimally decomposing a numerical semigroup with multiplicity $m$ into $m$-irreducible numerical semigroups.

\begin{defi}[Decomposition into $m$-irreducible numerical semigroups]
\label{def:3}
Let $S$ be a numerical semigroup with multiplicity $m$. Decomposing $S$ into $m$-irreducible numerical semigroups consists of finding a set of $m$-irreducible numerical semigroups $S_1, \ldots, S_{{\rm r} (S)}$ such that
 $S=S_1 \cap \cdots \cap S_{{\rm r} (S)}$ (This decomposition is always possible by Proposition \ref{prop:2}).

A minimal decomposition of $S$ into $m$-irreducible numerical semigroups is a decomposition with minimum {{\rm r} (S)} (minimal cardinality of the number of $m$-irreducible numerical semigroups involved in the decomposition).
\end{defi}

For a numerical semigroup $S$, the set of gaps of $S$, $\G(S)$, is the set $\Z_+\backslash S$ (that is finite by definition of numerical semigroup). We denote by $\g(S)$ the cardinal of that set, that is usually called the \emph{genus} of $S$. The \emph{Frobenius number} of $S$, $\F(S)$, is the largest integer not belonging to $S$.

Let $S$ be a numerical semigroup with multiplicity $m$. To decompose $S$ into $m$-irreducible numerical semigroups, we first need to know how to identify those $m$-irreducible numerical semigroups. In \cite{blanco-rosales10} it is proved that $S$ is $m$-irreducible if and only if it is maximal (w.r.t. the inclusion order) in the set of numerical semigroups with multiplicity $m$ and Frobenius number $\F(S)$. In \cite{rosales03} it is stated that a numerical semigroup, $S$, is irreducible if and only if $\g(S)=\left\lceil \dfrac{\F(S)+1}{2} \right\rceil$. Next, we recall two results in \cite{blanco-rosales10} that characterize, in terms of the genus and the Frobenius number, the set of $m$-irreducible numerical semigroups.

\begin{prop}[\cite{blanco-rosales10}]
\label{prop:4}
 A numerical semigroup with multiplicity $m$, $S$, is $m$-irreducible if and only if one of the following conditions holds:
\begin{enumerate}
\item If $\F(S)=\g(S)=m-1$ then $S=\{x \in \Z_+: x\ge m\} \cup \{0\}$.
\item If $\F(S) \in \{m+1, \ldots, 2m-1\}$ and $\g(S)=m$ then $S=\{x \in \Z_+: x\ge m, x\neq \F(S)\} \cup \{0\}$.
\item If $\F(S)>2m$ then  $S$ is an irreducible numerical semigroup. (In this case $\g(S)=\left\lceil \dfrac{\F(S)+1}{2} \right\rceil$).
\end{enumerate}
\end{prop}

\begin{cor}[\cite{blanco-rosales10}]
\label{cor:5}
 Let $S$ be a numerical semigroup with multiplicity $m$. Then, $S$ is $m$-irreducible if and only if $\g(S) \in \big\{m-1, m, \left\lceil \dfrac{\F(S)+1}{2} \right\rceil\big\}$.
\end{cor}

For a given numerical semigroup, $S$, our goal is to find a set of $m$-irreducible numerical semigroups whose intersection is $S$. Then,  we can restrict the search of these semigroups to the set of numerical semigroups that contain $S$. This set is called the set of \emph{oversemigroups} of $S$.

\begin{defi}[Oversemigroups]
\label{def:6}
Let $S$ be a numerical semigroup with multiplicity $m$. The set, $\mathcal{O}(S)$, of oversemigroups of $S$ is
$$
\mathcal{O}(S) = \{S'\mbox{ numerical semigroup}: S \subseteq S'\}
$$
The set, $\mathcal{O}_m(S)$, of oversemigroups of $S$ with multiplicity $m$ is $\mathcal{O}_m(S) = \{S ' \in \mathcal{O}(S): \m(S')=m\}$.
\end{defi}
Denote by $\mathcal{I}_m(S)$ the set of minimal $m$-irreducible numerical semigroups, with respect to the inclusion poset, in the set $\mathcal{O}_m(S)$.
From the set $\mathcal{I}_m(S)$ we can obtain,
a first decomposition of $S$ into $m$-irreducible numerical semigroup, although, in general, it is not minimal (see Example 27 in \cite{blanco-rosales10}).

\begin{lem}
 \label{lemma:7}
Let $S$ be a numerical semigroup with multiplicity $m$ and $\mathcal{I}_m(S)=\{S_1, \ldots, S_n\}$. Then $S=S_1 \cap \cdots \cap S_n$ is a decomposition of $S$ into $m$-irreducible numerical semigroups.
\end{lem}

Clearly, this basic decomposition is not ensured to be minimal since it may use redundant elements.

\begin{rem}
\label{rem:8}
Note that if $\hat{S}$ is a numerical semigroup with multiplicity $m$, by Proposition \ref{prop:4}, $\g(\hat{S}) = m-1$ if and only if $\hat{S}=\{0, m, \rightarrow\}$ ($\rightarrow$ denotes that every integer greater than $m$ belongs to $\hat{S}$).
  Hence, this $m$-irreducible numerical semigroup only appears in its own decomposition and in no one else.

This is due to the fact that $\hat{S}=\{0, m, \rightarrow\}$ is the maximal element in the set of numerical semigroups with multiplicity $m$, and then $\mathcal{O}_m(\hat{S})=\mathcal{I}_m(\hat{S})=\{\hat{S}\}$ (see \cite{blanco-rosales10} for further details).
\end{rem}

From now on, we assume that $S \neq \hat{S}=\{0, m, \rightarrow\}$ since by the above remark, the decomposition of $\hat{S}$ is trivial.

By Proposition \ref{prop:4} and Remark \ref{rem:8}, if $S\neq \hat{S}=\{0, m, \rightarrow\}$, its decomposition into $m$-irreducible numerical semigroups uses two types of numerical semigroups: those that are irreducible ($\g(S)=\left\lceil \dfrac{\F(S)+1}{2} \right\rceil$) and those that have genus equal the multiplicity of $S$.

 To refine the search of the elements in $\mathcal{I}_m(S)$, first, we need to introduce the notion of special gap.

\begin{defi}
\label{def:9}
Let $S$ be a numerical semigroup. The special gaps of $S$ are the elements in the following set:
$$
{\rm SG}(S) = \{ h \in {\rm G}(S): S \cup \{h\} \mbox{ is a numerical semigroup}\}
$$
where ${\rm G}(S)$ is the set of gaps of $S$.
\end{defi}
We denote by $\SG_m(S)=\{h \in \SG(S): h>m\}$. In \cite{blanco-rosales10}, the authors proved that $S$ is $m$-irreducible if and only if $\#\SG_m(S)\leq 1$ ($\#A$ stands for the cardinality of the set $A$). Moreover, $\SG_m(S)=\emptyset$ if and only if $S=\{0, m, \rightarrow\}$ (there are no gaps greater than $m$ in $S$).

Also, if we know the special gaps of a numerical semigroup, we can search for its decomposition by using the following result.

\begin{prop}[\cite{blanco-rosales10}]
\label{prop:11}
Let $S, S_1, \ldots, S_n$ be numerical semigroups with multiplicity $m$. $S=S_{1} \cap \cdots \cap S_{n}$ if and only if $\SG_m(S) \cap \left(\G(S_{1}) \cup \cdots \cup \G(S_{n})\right) = \SG_m(S)$.
\end{prop}

From the above proposition, even if the minimal $m$-irreducible numerical semigroups are known, $\mathcal{I}_m(S)=\{S_1, \ldots, S_m\}$,
 some of these elements may be discarded when looking for a minimal m-irreducible decomposition, by checking  if there are redundant elements in the intersection
 $\SG_m(S) \cap \left(\G(S_{1}) \cup \cdots \cup \G(S_{n})\right)$.

Then, the key is to choose elements in $\mathcal{I}_m(S)$ that minimally cover the special gaps of $S$. To this for, we may solve a problem fixing each of the special gaps to be covered. Note that an upper bound  of the number of problems to be solved is the number of special gaps of a numerical semigroup that is bounded above by $m-1$ (see \cite{Ro-GS09}).

\begin{lem}
\label{lemma:12}
Let $S \neq \{0, m, \rightarrow\}$ be a numerical semigroup with multiplicity $m$, and $h \in \SG_m(S)$. Then, there exists a minimal decomposition of $S$ into $m$-irreducible numerical semigroups, $S=S_1 \cap \cdots \cap S_n$ such that, either $h = \F(S_i)$ for some $i$ or $h \not\in S_i$ for some $i$ such that there exists $h'\in \SG_m(S_i)$ with $\F(S_i) = h'>h$.
\end{lem}

\begin{proof}
By Proposition \ref{prop:2}, there exists a minimal decomposition of $S$ into $m$-irreducible numerical semigroup, $S=S_1 \cap \cdots \cap S_k$. By applying Proposition \ref{prop:11}, this decomposition must verify that $\SG_m(S) \cap \left(\G(S_{1}) \cup \cdots \cup \G(S_{n})\right) = \SG_m(S)$. Each special gap $h\in \SG_m(S)$ must be in $\G(S_i)$ for some $i=1, \ldots, n$. Assume that $h \neq \F(S_i)$ and that for all $h' \in \SG_m(S_i)$ with $h'>h$, $\F(S_i)\neq h'$. Then, $S_i' = S_i \cup \{\F(S_i)\}$ is a $m$-irreducible numerical semigroup such that  $\SG_m(S) \cap \left(\G(S_{1}) \cup \cdots \G(S_i') \cdots \cup \G(S_{n})\right) = \SG_m(S)$. Then, we have obtained a different minimal decomposition. (Note that it has the same number of terms than the original one.)

By repeating this procedure for each $h\in \SG_m(S)$ while it is possible, we find a minimal decomposition of $S$ fulfilling the conditions of the lemma.
\end{proof}

\section{The Kunz-coordinates vector}
\label{sec:2}
The analysis done through this paper uses mathematical programming tools to solve the problem of decomposing a numerical semigroup
into $m$-irreducible numerical semigroups.
For the sake of translating the problem to a discrete optimization problem, we use an alternative encoding of numerical semigroups different from the system of generators.
We identify each numerical semigroup with multiplicity $m$ with a nonnegative integer vector with $m-1$ coordinates,
 where $m$ is the multiplicity of the semigroup. To describe this identification we first need to give the notion of Ap\'ery set of a numerical semigroup with respect to its multiplicity

\begin{defi}
\label{def:13}
Let $S$ be a numerical semigroup with multiplicity $m$. The \emph{Ap\'ery set} of $S$ with respect to $m$ is the set $\Ap(S,m) = \{s
\in S :  s - m \not\in S\}$.
\end{defi}

However we are interested in the following characterization  of the Ap\'ery set (see \cite{Ro-GS09}): Let $S$ be a numerical semigroup with multiplicity $m$, then $\Ap(S,m) = \{0 = w_0, w_1, \ldots,  w_{m -
1}\}$, where $w_i$ is the least
element in $S$ congruent with $i$ modulo $m$, for $i=1, \ldots, m-1$.

Moreover, the set $\Ap(S,m)$ completely determines $S$, since $S =
\langle \Ap(S,m) \cup \{m\} \rangle$ (see \cite{rosales02-london}), and then, we can identify $S$ with its Ap\'ery set with respect to its multiplicity. Moreover, the set $\Ap(S,m)$
contains, in general, more information than an arbitrary system of
generators of $S$. For instance, Selmer in \cite{selmer77} gives the formulas,
$\g(S)=\frac{1}{m}\left(\sum_{w \in \Ap(S, m)} w\right) -
\frac{m-1}{2}$ and $\F(S) = \max(\Ap(S,m)) - m$. Moreover, one can test if a nonnegative integer $s$ belongs to $S$ by checking if $w_{s\pmod m} \leq s$.
The notion of Ap\'ery set is also given when we consider any $n\in S$ instead of $m$, rewriting then the definition adequately  (see \cite{Ro-GS09}). Moreover, the smallest Ap\'ery set is $\Ap(S, m)$.

We consider an slightly but useful modification of the Ap\'ery set that we call the \emph{Kunz-coordinates vector}.

\begin{defi}[Kunz-coordinates]
\label{def:14}
Let $S$ be a numerical semigroup with multiplicity $m$. If $\Ap(S, m)=\{w_0=0, w_1, \ldots, w_{m-1}\}$, with $w_i$ congruent with $i$ modulo $m$, the \emph{Kunz-coordinates vector} of $S$ is the vector $x \in \Z^{m-1}_+$ with components $x_i = \frac{w_i-i}{m}$ for $i=1, \ldots, m-1$.

We say that $x\in \Z^{m-1}_+$ is a Kunz-coordinates vector (or Kunz-coordinates, for short) if there exists a numerical semigroup whose Kunz-coordinates vector is $x$.
\end{defi}

From the Kunz-coordinates we can reconstruct the Ap\'ery set. If $x \in \Z^{m-1}_+$  is the Kunz-coordinates vector of $S$, $\Ap(S, m) = \{mx_i + i: i = 1, \ldots, m-1\} \cup \{0\}$. Consequently, $S$ can be reconstructed from its Kunz-coordinates.

The Kunz-coordinates vectors have been implicitly used in \cite{kunz} and \cite{rosales02-london} to characterize numerical semigroups with fixed multiplicity, and used in \cite{counting} to count numerical semigroups with a given genus.

Furthermore, if $S$ is a numerical semigroup with multiplicity $m$ and $x \in \Z^{m-1}_+$ are its Kunz-coordinates, from Selmer's formulas, it is easy to compute its genus and its Frobenius number as follows:
\begin{itemize}
 \item $\g(S) = \dsum_{i=1}^{m-1} x_i$.
\item $\F(S) = \max_i\{mx_i+i\}-m$. (Clearly, if the maximum is reached in the $i$-th component, $\F(S) \equiv i \pmod m$)
\end{itemize}

The following result that appears in \cite{rosales02-london} allows us to manipulate numerical semigroups with multiplicity $m$ as integer points inside a polyhedron.

\begin{theo}[Theorem 11 in \cite{rosales02-london}]
\label{theo:15}
Each numerical semigroup is one-to-one identified with its Kunz-coordinates.

Furthermore, the Kunz-coordinates vectors of the set of numerical semigroups with multiplicity $m$ is the set of solutions of the following system of diophantine inequalities:
\begin{align*}
x_i  \geqslant&1 & \mbox{for all $i \in \{1, \ldots, m-1\}$,}\\
x_i+x_j-x_{i+j} \geqslant& 0 & \mbox{for all $1 \leqslant i \leqslant j \leqslant m-1$, $i+j \leqslant m-1$,}\\
x_i+x_j-x_{i+j-m}  \geqslant& -1 &\mbox{for all $1 \leqslant i \leqslant j \leqslant m-1$, $i+j > m$},\\
x_i \in \Z_+ & &\mbox{for all $i \in \{1, \ldots, m-1\}$}.
\end{align*}
\end{theo}

From Theorem \ref{theo:15} and Selmer formulas, we can identify all the numerical semigroups (in terms of their Kunz-coordinates vector) with multiplicity $m$, genus $g$ and Frobenius number $F$  with the solutions of this system of diophantine inequalities:
\begin{align*}
x_i  \geqslant&1 & \mbox{for all $i \in \{1, \ldots, m-1\}$,}\nonumber\\
x_i+x_j-x_{i+j} \geqslant& 0 & \mbox{for all $1 \leqslant i \leqslant j \leqslant m-1$, $i+j \leqslant m-1$,}\nonumber\\
x_i+x_j-x_{i+j-m}  \geqslant& -1 &\mbox{for all $1 \leqslant i
\leqslant j \leqslant m-1$, $i+j > m$,}\\
\dsum_{i=1}^{m-1} x_i =& g &\\
F  =& \max_i\;\{mx_i+i\}- m ,\nonumber\\
x_i &\in \Z_+ &\mbox{for all $i \in \{1, \ldots, m-1\}$}.\nonumber
\end{align*}

From the above formulation and Corollary \ref{cor:5}, the set of $m$-irreducible numerical semigroups is completely determined by the solutions of the following disjunctive diophantine system of inequalities and equations:
\begin{align*}
x_i  &\geqslant 1 & \mbox{for all $i \in \{1, \ldots, m-1\}$,}\nonumber\\
x_i+x_j-x_{i+j} &\geqslant 0  &\mbox{for all $1 \leqslant i \leqslant j \leqslant m-1$, $i+j \leqslant m-1$,}\nonumber\\
x_i+x_j-x_{i+j-m}  &\geqslant -1 &\mbox{for all $1 \leqslant i
\leqslant j \leqslant m-1$, $i+j > m$,}\\
\dsum_{i=1}^{m-1}& x_i \in \{m-1, m, \max_i\;\{mx_i+i\}- m\} \\
x_i \in& \Z_+ &\mbox{for all $i \in \{1, \ldots, m-1\}$}.\nonumber
\end{align*}

The following result characterizes the set of oversemigroups of a numerical semigroup, in term of its Kunz-coordinates vector.
\begin{prop}
\label{prop:16}
 Let $S$ be a numerical semigroup with multiplicity $m$ and $x \in \Z^{m-1}_+$ its Kunz-coordinates. Then, the set of Kunz-coordinates vectors of oversemigroups of $S$ with multiplicity $m$ is:
$$
\{x' \in \Z^{m-1}_+: x' \mbox{ is a Kunz-coordinates vector and } x' \leq x\}
$$
where $\leq$ denotes the componentwise order in $\Z^{m-1}$.
\end{prop}
\begin{proof}
Let $S' \in \mathcal{O}_m(S)$, and $\Ap(S', m) = \{0, w_1', \ldots, w'_{m-1}\}$. Let $\Ap(S,m) = \{0, w_1, \ldots, w_{m-1}\}$. The $i$th element in the Ap\'ery set is characterized of being the minimum element in the semigroup that is congruent with $i$ modulo $m$. Thus, $w_i' \leq w_i$ for all $i=1,  \ldots, m-1$, since $S \subseteq S'$. Then $x'_i = \frac{w'_i-i}{m} \leq \frac{w_i-i}{m} = x_i$ for all $i=1, \ldots, m-1$. Hence, $x' \leq x$.
\end{proof}

Therefore, the oversemigroups of $S$ can be identified with the ``\emph{undercoordinates}'' of its Kunz-coordinates.

For the ease of presentation, we identify a numerical semigroup with multiplicity $m$ with an integer vector with $m-1$ coordinates, its Kunz-coordinates. All the notions previously given for numerical semigroups are adapted conveniently by using the following notation. If $S$ is a numerical semigroup and $x\in \Z^{m-1}$ is its Kunz-coordinates vector, we denote:
\begin{itemize}
\item $\m(x)=\m(S)=m$ (Multiplicity of $x$).
\item $\F(x)=\F(S)$ (Frobenius number).
\item $\G(x)=\G(S)=\{n \in \Z: mx_{n \pmod m}+ n\pmod m > n\}$ (Gaps of $x$).
\item $\g(x)=\g(S)$ (Genus of $x$).
\item $\SG(x)=\SG(S)$ (Special Gaps of $x$).
\item $\SG_m(x)=\SG_m(S)$ (Special Gaps greater than $m$ of $x$).
\item $\mathcal{U}_m(x)= \{x' \in \Z^{m-1}: \text{ $x'$ is a Kunz-coordinates vector and } x'\le x\}$ (Undercoordinates of $x$). Observe that if $x$ is the Kunz-coordinates vector of $S$, $x' \in \mathcal{U}_m(x)$ is univocally identified with an element $S'\in \mathcal{O}_m(S)$ (Proposition \ref{prop:16}).
\item $\Ap(x) = \Ap(S, m) = \{0\} \cup \{mx_i + i: i=1, \ldots, m-1\}$ (Ap\'ery set of $x$).
\end{itemize}
Note that all the above indices and sets can be computed by using only the Kunz-coordinates vector of the semigroup.

As assumed above, we consider that $S\neq \{0, m, \rightarrow\}$. In terms of the Kunz-coordinates, it is equivalent to say that $x\neq (1, \ldots, 1)\in \Z^{m-1}_+$ (or $\dsum_{i=1}^{m-1} x_i \geq m$).

By Corollary \ref{cor:5} we say that a Kunz-coordinates vector, $x \in \Z_+^{m-1}$ is $m$-irreducible if $\g(x) \in \{m, m-1, \left\lceil\frac{\F(x)+1}{2}\right\rceil\}$. Furthermore, we say that $x$ is irreducible if $\g(x)=\left\lceil\frac{\F(x)+1}{2}\right\rceil$.
  Hence, every irreducible Kunz-coordinates vector in $\Z^{m-1}_+$ is $m$-irreducible, but the converse is not true in general.

We also say that a set of Kunz-coordinates vectors, $D=\{x^1, \ldots, x^k\} \subseteq \Z^{m-1}_+$, is a decomposition of $x\in \Z^{m-1}_+$ into $m$-irreducible Kunz-coordinates vectors if the semigroups associated with the elements in $D$ give a decomposition into $m$-irreducible numerical semigroups of the semigroup identified with $x$. Equivalently, by Proposition \ref{prop:11}, $D$ is a decomposition of $x\in \Z^{m-1}_+$ into $m$-irreducible Kunz-coordinates vectors if $x^i$ is an $m$-irreducible Kunz-coordinates vector and $\SG_m(x) = \SG_m(x) \cap \left(\G(x^1) \cup \cdots \cup \G(x^k)\right)$.

Then, a minimal decomposition $x\in \Z^{m-1}_+$ into $m$-irreducible Kunz-coordinates is a decomposition into $m$-irreducible Kunz-coordinates, $D=\{x^1, \ldots, x^k\} \subseteq \Z^{m-1}_+$, with minimum cardinality.

We define
$$
\begin{array}{rl} \mathcal{I}_m(x)= & \{x' \in \mathcal{U}_m(x): x' \text{is $m$-irreducible and } \not\exists \text{ a $m$-irreducible Kunz-coordinates}\\ & \text { vector }  x^* \in \mathcal{U}_m(x)  \text{ such that } x^* \ge x'\} \end{array}
$$
$\mathcal{I}_m(x)$ is one-to-one identified with $\mathcal{I}_m(S)$.

\section{\emph{m}-irreducible Kunz-coordinates vectors}
\label{sec:3}
In this section we give necessary and sufficient conditions for a undercoordinate of a Kunz-coordinates vector to be $m$-irreducible.

Let $x \in \Z^{m-1}_+$ be a Kunz-coordinates vector. By the above definition, a Kunz-coordinates vector $x' \in \mathcal{U}_m(x)$ if and only if there exists $y \in \Z^{m-1}_+$ such that $x'+ y = x$.

By applying Theorem \ref{theo:15} to $x' = x-y$, $y \in \Z^{m-1}_+$ must verify the following inequalities:
\begin{align*}
y_i \leqslant x_i - 1 & \mbox{ for all $i \in \{1, \ldots, m-1\}$,}\\
y_i + y_j - y_{i+j} \leqslant x_i + x_j - x_{i+j} & \mbox{ for all $1 \leqslant i \leqslant j \leqslant m-1$, $i+j \leqslant m-1$,}\nonumber\\
y_i + y_j - y_{i+j} \leqslant x_i + x_j - x_{i+j} + 1 &\mbox{ for all $1 \leqslant i \leqslant j \leqslant m-1$, $i+j > m$}
\end{align*}

Actually, if we are searching for those $x'=x-y$ that are identified with a set of $m$-irreducible undercoordinates decomposing $x$, we can restrict, by Corollary \ref{cor:5}, to consider those with genus $m$, $m-1$ and $\left\lceil \dfrac{\F(x) +1}{2} \right\rceil$. Therefore, $y$ must be a solution of the following system:
\begin{align}
y_i \leqslant x_i - 1 & \mbox{ for all $i \in \{1, \ldots, m-1\}$,}\nonumber\\
y_i + y_j - y_{i+j} \leqslant x_i + x_j - x_{i+j} & \mbox{ for all $1 \leqslant i \leqslant j \leqslant m-1$, $i+j \leqslant m-1$,}\nonumber\\
y_i + y_j - y_{i+j} \leqslant x_i + x_j - x_{i+j} + 1 &\mbox{ for all $1 \leqslant i
\leqslant j \leqslant m-1$, $i+j > m$},\label{polytope:x}\tag{${\rm P}^m(x)$}\\
\dsum_{i=1}^{m-1} y_i \in M(x, y),\label{eq:disj}\\
y \in \Z^{m-1}_+.\nonumber
\end{align}
where $M(x, y) = \{\dsum_{i=1}^{m-1} x_i -m, \dsum_{i=1}^{m-1} x_i -m + 1, \dsum_{i=1}^{m-1} x_i - \left\lceil \dfrac{\max_i\{m(x_i-y_i) + i\} - m +1}{2} \right\rceil\}$.

Recall that the Kunz-coordinates vector $(1, \ldots, 1) \in \Z^{m-1}_+$ is not considered because it corresponds to $S=\{0, m, \rightarrow\}$ that is  $m$-irreducible, and then, its minimal decomposition is itself (Remark \ref{rem:8}). Clearly, these coordinates are the unique solution of the above system when constraint \eqref{eq:disj} is $\dsum_{i=1}^{m-1} y_i= \dsum_{i=1}^{m-1} x_i -m$.

In the next subsections we analyze the remaining two cases for the disjunctive constraint \eqref{eq:disj}.

\subsection{\emph{m}-irreducible undercoordinates that are irreducible}
Let $x \in \Z^{m-1}_+$ be a Kunz-coordinates vector. First, we want to find those $m$-irreducible undercoordinates of $x$ that are also irreducible. Then, in system \eqref{polytope:x}, equation \eqref{eq:disj} is
\begin{equation}
\label{eq:irreducible}
\dsum_{i=1}^{m-1} y_ i = \dsum_{i=1}^{m-1} x_i - \left\lceil \dfrac{\max_i\{m(x_i-y_i) + i\} - m +1}{2} \right\rceil.
\end{equation}

Denote now by ${\rm H}_k^m(x)= \{y \in \R^{m-1}: \max_i\{m(x_i-y_i)+i\} = m(x_k-y_k) + k\}$, and by ${\rm P}_k^m(x) = {\rm P}(x) \cap {\rm H}_k^m(x)$ for all $k=1, \ldots, m-1$. Note that ${\rm H}_k^m(x)$ is the hyperplane in $\R^{m-1}$ where the Frobenius number of $x-y$ is reached in the $kth$ component (recall that $\F(x)=\max\{mx_i+i\} - m$), that is, $\F(x-y) = m(x_k-y_k)+k-m$.

With this assumptions, ${\rm P}_k^m(x)$ can be described by the following system of inequalities:
{\begin{align}
\;y_i \leqslant& x_i - 1  \mbox{ for all $i \in \{1, \ldots, m-1\}$,}\nonumber\\
\,y_i +& y_j - y_{i+j} \leqslant x_i + x_j - x_{i+j}  \mbox{ for all $1 \leqslant i \leqslant j \leqslant m-1$, $i+j \leqslant m-1$,}\nonumber\\
\,y_i +& y_j - y_{i+j} \leqslant x_i + x_j - x_{i+j} + 1 \mbox{ for all $1 \leqslant i
\leqslant j \leqslant m-1$, $i+j > m$,}\label{polytope:x0}\tag{${\rm P}_k^m(x)$}\\
\,\dsum_{i=1}^{m-1}& y_i = \dsum_{i=1}^{m-1} x_i - \left\lceil \dfrac{m(x_k-y_k) + k - m +1}{2} \right\rceil,\nonumber\\
\,y \in& \Z^{m-1}_+.\nonumber
\end{align}}

Or equivalently (using that $z \leq \lceil z \rceil < z+1$ for any $z \in \R$) by a linear system of inequalities as:

{\begin{align}
y_i& \leqslant x_i - 1 & \mbox{ for all $i \in \{1, \ldots, m-1\}$,}\nonumber\\
y_i& + y_j - y_{i+j} \leqslant x_i + x_j - x_{i+j} & \mbox{ for all $1 \leqslant i \leqslant j \leqslant m-1$, $i+j \leqslant m-1$,}\nonumber\\
y_i& + y_j - y_{i+j} \leqslant x_i + x_j - x_{i+j} + 1 &\mbox{ for all $1 \leqslant i
\leqslant j \leqslant m-1$, $i+j > m$,}\label{polytope:x2}\tag{${\rm P}_k^m(x)$}\\
2\,&\dsum_{i=1}^{m-1} y_i - my_k \geqslant  2\dsum_{i=1}^{m-1} x_i - mx_k - k +m -2,\nonumber\\
2\,&\dsum_{i=1}^{m-1} y_i - my_k \leqslant  2\dsum_{i=1}^{m-1} x_i - mx_k - k +m -1,\nonumber\\
y\;&\in \Z^{m-1}_+.\nonumber
\end{align}}

\subsection{\emph{m}-irreducible undercoordinates with genus \emph{m}}

In what follows, we describe the second type of $m$-irreducible undercoordinates of $S$, those with genus $m$.

Denote by ${\rm HG}^m(x) = \{y \in \R^{m-1}: \dsum_{i=1}^{m-1} y_i = \dsum_{i=1}^{m-1} x_i - m\}$ and ${\rm P}^m_m(x) = {\rm P}^m(x) \cap {\rm HG}^m(x)$. This set is described by the following system of diophantine inequalities:
\begin{align}
y_i \leqslant x_i - 1 & \mbox{ for all $i \in \{1, \ldots, m-1\}$,}\nonumber\\
y_i + y_j - y_{i+j} \leqslant x_i + x_j - x_{i+j} & \mbox{ for all $1 \leqslant i \leqslant j \leqslant m-1$, $i+j \leqslant m-1$,}\nonumber\\
y_i + y_j - y_{i+j} \leqslant x_i + x_j - x_{i+j} + 1 &\mbox{ for all $1 \leqslant i
\leqslant j \leqslant m-1$, $i+j > m$}\label{polytope:xm}\tag{${\rm P}^m_m(x)$}\\
\dsum_{i=1}^{m-1} y_i = \dsum_{i=1}^{m-1} x_i - m\label{eq:m}\\
y \in \Z^{m-1}_+.\nonumber
\end{align}

The solutions of system \eqref{polytope:xm} are easily identified by the few possible choices for the solutions of equation \eqref{eq:m} (the integer vector $x-y \in \Z^{m-1}$ has positive coordinates and the sum of them must be $m$). Actually, the entire set of solutions of \eqref{polytope:xm} is:
$$
\{x-(1, \ldots, 1)-\e_j: x_j \geqslant 2, j=1, \ldots, m-1\} \subseteq \Z^{m-1}_+
$$
where $\e_j$ is the $j$th unit vector in $\Z^{m-1}_+$.

Then, the set of $m$-irreducible undercoordinates of $x$ with genus $m$ is given by the set $\{(1, \ldots, 1) + \e_j: x_j \geqslant 2, j=1, \ldots, m-1\} \subseteq \Z^{m-1}_+$.
\section{Decomposing into \emph{m}-irreducible numerical semigroups}
\label{sec:4}
In the section above we characterize the $m$-irreducible undercoordinates of a Kunz-coordinates vector $x\in \Z_+^{m-1}$.
 In what follows, we use these characterizations to find a decomposition of $x$ into $m$-irreducible Kunz-coordinates vectors.
 We first give some decompositions that are not minimal in general by enumerating the whole set of solutions of the systems
\eqref{polytope:x2} and \eqref{polytope:xm}. After that we provide a multiobjective integer programming model to obtain the set
of minimal elements in $\mathcal{I}_m(x)$. We prove that this model is equivalent to enumerate the entire set of optimal solutions
of some single-objective integer programming problems. Thus, a minimal decomposition can be obtained from the former set of solutions
by solving a set covering problem. Finally, we propose a heuristic methodology based on the abovementioned exact approach to obtain a (minimal) decomposition of $x$ into $m$-irreducible Kunz-coordinates vectors.

 As a consequence of Corollary \ref{cor:5} and the comments above we obtain the following result that states how to get a decomposition into $m$-irreducible Kunz-coordinates vectors by solving several systems of diophantine inequalities.
\begin{prop}
\label{prop:17}
Let $x\in \Z_+^{m-1}$ be a Kunz-coordinates vector. Any decomposition of $x$ into $m$-irreducible Kunz-coordinates vectors is given by some elements in the form $x-y$ where $y$ belongs to the union of the solutions of the systems ${\rm P}^m_1(x), \ldots, {\rm P}^m_{m-1}(x)$ and ${\rm P}^m_m(x)$.
\end{prop}

\begin{rem}
\label{rem:18}
 Note that the whole set of solutions of ${\rm P}^m_1(x), \ldots, {\rm P}^m_{m-1}(x)$ and ${\rm P}^m_m(x)$ gives a decomposition into $m$-irreducible numerical semigroups of the semigroup $S$ identified with $x$. It is the maximal decomposition since it has the maximum possible number of $m$-irreducible Kunz-coordinates involved, all the $m$-irreducible undercoordinates of $x$.
\end{rem}

In the following we give a methodology to compute minimal decompositions. The main idea is to choose, adequately, solutions of the systems ${\rm P}^m_1(x), \ldots, {\rm P}^m_{m-1}(x)$ and ${\rm P}^m_m(x)$.

The first step to select decompositions  that are minimal with respect to the inclusion ordering is to find the minimal elements in the set of $m$-irreducible undercoordinates of a Kunz-coordinates vector $x$. This fact can be formulated as a multiobjective integer programming problem as stated in the following result.

\begin{theo}
\label{theo:19}
Let $x\in \Z^{m-1}_+$ be a Kunz-coordinates vector. The Kunz-coordinates vectors of the elements in $\mathcal{I}_m(x)$ are in the form $x-\hat{y}$ where $\hat{y}$ is a nondominated solution of any of the following multiobjective integer linear programming problems.
\begin{equation}
\label{mop}\tag{${\rm MIP}^m_k(x)$}
v-\min \;(y_1, \ldots, y_{m-1}) \text{ s.t. } y \in {\rm P}_k^m(x), \quad \text{ for } k=1, \ldots, m-1, m.
\end{equation}

\end{theo}
\begin{proof}
Let $x'$ be an element in $\mathcal{I}_m(x)$. Then, $x'=x-y'$ for some $y \in \Z^{m-1}_+$. If $k=\F(x') \pmod m$, then, $\F(x') = mx_k'+k-m$.
Since $x'$ is an $m$-irreducible undercoordinate of $x$ with the above Frobenius number, either $y' \in {\rm P}_k^m(x)$ (if $\F(x')>2m$) or $y' \in {\rm P}^m_m(x)$ (if $\F(x') < 2m$). Suppose that there is a nondominated solution, $\hat{y}$, of ${\rm MIP}^m_k(x)$  (resp. ${\rm MIP}^m_m(x)$) dominating $y'$. Then, we can find $\hat{x} = x - \hat{y}$, with $\hat{y}$ nondominated solution of ${\rm MIP}^m_k(x)$  (resp. ${\rm MIP}^m_m(x)$) such that $\hat{y} \leq y'$ and $\hat{y} \neq y'$. Then, $\hat{x} \ge x'$ and $x'\neq \hat{x}$, and consequently, we have found an $m$-irreducible maximal Kunz-coordinates in  $\mathcal{I}_m(x)$ such that $\hat{x} \ge x'$ and $x'\neq \hat{x}$, contradicting the maximality of $x'$.
\end{proof}

Note that, $\Gamma$, the union of the nondominated solutions of ${\rm MIP}^m_1(x)$, \ldots, ${\rm MIP}^m_m(x)$ contains $\mathcal{I}_m(x)$,
but it may contain nondominated solutions of ${\rm MIP}^m_k(x)$ that dominate some nondominated solution of
 ${\rm MIP}^m_j(x)$, if $k\neq j$. Thus, $\Gamma$ may contain coordinates vectors that dominate one another what would lead to non minimal decompositions into $m$-irreducible Kunz-coordinates vectors.

The key to get minimal decompositions into $m$-irreducible Kunz-coordinates follows by applying Lemma \ref{lemma:12}. Therefore, we need to address the question about how to compute $\SG_m(x)$. Algorithm \ref{alg:sg} shows the way of computing the special gaps greater than the multiplicity of a Kunz- coordinates vector. This algorithm is based in the following result. There,  ${\km} (n)=n \mod m$ stands for the nonnegative integer remainder of dividing $n$ by $m$, i.e., ${\km} (n) = n \mod m$.

\begin{theo}
Let $x\in \Z^{m-1}_+$ be a Kunz-coordinates vector and $m<h \in \N$. Then, $h \in \SG_m(x)$ if and only if $h=(x_{{\km} (h)}-1)+{\km} (h)$ and such that $x_{{\km} (h)}+x_j>x_{{\km} ({\km} (h)+j)}-\gamma_{{\km} (h), j}$ for all $j=1, \ldots, m$ with ${\km} (h)+j \neq m$ and $2h \geq mx_{{\km} (2h)} + {\km} (2h)$; and where $\gamma_{ij}=\left\{\begin{array}{rl} 1 & \mbox{if $i+j>m$}\\
0 & \mbox{ otherwise}\end{array}\right.$ for all $i, j =1, \ldots, m-1$.
\end{theo}

\begin{proof}
The elements in $\SG_m(x)$ are those elements fulfilling the following conditions (see \cite{blanco-rosales10}):
\begin{itemize}
\item $h=w_i - m$, where $w_i \in \Ap(x)$, for some $i=1, \ldots, m-1$.
\item $w_i-w_j \not\in \Ap(x)$ for all $w_j \in \Ap(x)$, $w_j\neq w_i$.
\item $2h \geq w_{{\km} (2h)}$
\end{itemize}
By the identification of Kunz-coordinates vectors and the elements in the Ap\'ery set, the first conditions
are translated in $h=mx_i+i - m=m(x_i-1)+i$. The second set of conditions consist of checking for each $j\neq i$
 if $w_i-w_j = mx_i+i-mx_j-j \not\in \{0\} \cup \{mx_k+k: k=1, \ldots m-1\}$. Note that if $mx_i+i-mx_j-j = mx_k+k$
 for some $k$, then, ${\km} (k)={\km} (i-j)$, so, if $w_i-w_j$ is an element in $\Ap(x)$ the unique possible choice is $w_{{\km} (i-j)}$.
 Now, if $i>j$, then ${\km} (i-j)=i-j$, and the condition is the same as checking if $mx_i + i - mx_j - j \neq mx_{i-j} + i-j$,
equivalently, if $x_i + x_{i-j} \neq  x_j$. Since $x$ is a Kunz-coordinates vector, by Theorem \ref{theo:15}, $x_i + x_{i-j} \geq x_{j}$,
so checking that those elements are different is the same that $x_i + x_{i-j} > x_{j}$. Clearly, by changing indices, it is
that $x_i + x_j > x_{i+j} + \gamma_{ij}$ (in this case $i+j>m$). The case when $i<j$ is analogous but taking into account that in that case ${\km} (i-j) = i-j + m$.

The third conditions is direct from the algorithm given in \cite{blanco-rosales10} to compute $\SG_m(x)$.
\end{proof}

The above theorem is used to compute the set $\SG_m(x)$ for any Kunz-coordinates vector $x\in \Z^{m-1}_+$ as shown in Algorithm \ref{alg:sg}.

\begin{algorithm2e}[h]
\label{alg:sg}

\SetKwInOut{Input}{Input}
\SetKwInOut{Output}{Output}
\Input{A Kunz-coordinates vector $x \in \Z^{m-1}_+$.}

Compute $M_1=\{m(x_i-1)+i: x_i+x_j>x_{i+j}, \text{ for all } j \text{ with } i+ j<m\}$ and $M_2=\{m(x_i-1)+i: x_i+x_j>x_{i+j-m}-1, \text{ for all } j \text{ with } i+ j>m\}$.

$ $

\Output{$\SG_m(x)=\{z \in M_1 \cap M_2: z>m \text{ and } 2z \ge m\,x_{{\km} (2z) } + {\km} (2z) \}$.}
\caption{Computing the special gaps greater than the multiplicity of a Kunz-coordinates vector.}
\end{algorithm2e}
Note that the complexity of Algorithm \ref{alg:sg} is $O(m^2)$.

From Algorithm \ref{alg:sg} and the Kunz-coordinates vector of a numerical semigroup, we obtain the following useful result.

\begin{prop}
\label{prop:20}
Let $x \in \Z^{m-1}_+$ be a Kunz-coordinates vector, $y \in \Z^{m-1}_+$ and $h \in \SG_m(x)$. If $x-y$ is a
 undercoordinate of $x$, then, $h \in \G(x-y)$ if and only if $y_{{\km} (h)} =0$.
Furthermore, $\F(x-y)$ is the unique element in $\{h \in \SG_m (x): {{\km}} (h)  = \max\{i \in \{1, \ldots, m-1\}: y_i=0\}\}$.
\end{prop}
\begin{proof}
Since $h\in \SG_m(x)$, by Algorithm \ref{alg:sg}, $h=m(x_{{\km} (h) } - 1) + {\km} (h) $.

If $h \in \G(x-y)$ then, $m(x_{{\km} (h) }-y_{{\km} (h) }) + {\km} (h)  \ge h+1 = m(x_{{\km} (h)} - 1) + {\km} (h) + 1$, that is, $y_{{\km} (h)} \leq \dfrac{m-1}{m} < 1$, and then $y_{{\km} (h) }=0$ because $y_i\ge 0$ for all $i=1, \ldots, m-1$.

Conversely, if $y_{{\km} (h) }=0$, then,  $m(x_{{\km} (h) }-y_{{\km} (h) }) + {\km} (h)  = mx_{{\km} (h) } + {\km} (h)  \geq h+1$ since $h$ is an special gap of $x$, and then, in particular, a gap of $x$. Thus, $h \in \G(x-y)$.
\end{proof}

By Proposition \ref{prop:11}, for each $h \in {\rm SG}_m(x)$ we are looking among our solution, $y$, for one that holds $h \in \G(x-y)$. This is equivalent, by Proposition \ref{prop:20}, to search for those with $y_{{\km} (h) }=0$. Then, from all the minimal $m$-irreducible numerical oversemigroups of $S$, we only need for the minimal decomposition, those that do not contain the special gaps of $S$. The following result even shrink further this search.

\begin{lem}
\label{lemma:20}
Let $x\in   \Z^{m-1}_+$ be a Kunz-coordinates vector and $h \in \SG_m(x)$. Then, every nondominated solution of \eqref{mop}, $y$, has $y_{{\km} (h)}=0$, and then, $\F(x-y)=h$. Moreover, $y$ is the solution with the minimum sum of its coordinates (length).
\end{lem}
\begin{proof}
By Algorithm \ref{alg:sg}, $h= mx_{{\km} (h) } +{{\km} (h)}-m$. Furthermore, $h \in \G(x')$ for some $x'=x-y$ in the decomposition, so $m(x_{{\km} (h) }-y_{{\km} (h) })+{{\km} (h) } \ge h+1 = mx_{{\km} (h) } + {\km} (h)  -m + 1$. Then, $y_{{\km} (h) } \leq 1 - \frac{1}{m} < 1$, being then $y_{{\km} (h) }=0$.

Then, we have a feasible solution of \ref{mop} with $y_{{\km} (h)}=0$. Therefore, for each nondominated solution, $\hat{y}$, dominating $y$, i.e., $\hat{y} \leq y$ and $y\neq \hat{y}$. The former implies that $\hat{y}_{{\km} (h) }=0$.

Furthermore, $\F(x-\hat{y}) = m(x_k-y_k) + k - m = mx_k + k -m = h$.

Since any feasible solution, $y'$, of \eqref{mop} must hold $\dsum_{i=1}^{m-1} y'_i = \dsum_{i=1}^{m-1} x_i - \left\lceil \dfrac{m(x_k-y'_k) + k - m +1}{2} \right\rceil$, then, $\dsum_{i=1}^{m-1} y_i' \geqslant \dsum_{i=1}^{m-1} x_i - \left\lceil \dfrac{mx_k + k - m +1}{2} \right\rceil = \dsum_{i=1}^{m-1} y_i$, so $y$ has minimum length, and no one else has this length.
\end{proof}

By the above result we know that, if we fix a special gap, $h$,
 a nondominated solution of \ref{mop} with minimum length can be computed by fixing the value of $y_{{\km} (h)}$.
Then, moving through all the special gaps in $\SG_m(x)$ and fixing each one of them in \ref{mop}, we can obtain at least $\#SG_m(x)$ nondominated solutions giving a decomposition of $x$ into $m$-irreducible Kunz-coordinates.

Therefore, an upper bound on the number of elements in any decomposition is the number of special gaps greater than the multiplicity of the semigroup. Thus, for each problem ${\rm P}_k^m(x)$ we can add the constraint requiring that $h$ is a gap of the Kunz-coordinates vector, for each $h\in \SG_m(x)$, i.e., $y_{{\km} (h)}=0$. Then, for each $h \in \SG_m(x)$ and $k \in \{1, \ldots, m\}$ we need to solve the following multiobjective problem:

\begin{equation}
\label{mopoly:x}
\begin{array}{lll}
v-\min &(y_1, \ldots, y_{m-1}) &\\
s.t.&&\\
&y_{{\km} (h)}=0\\
&y \in {\rm P}_k^m(x)\\
\end{array}\tag{${\rm MIP}^m(x,h)$}
\end{equation}

\begin{rem}
\label{rem:21}
By Lemma \ref{lemma:20}, it is enough to search for those $m$-irreducible Kunz-coordinates with Frobenius numbers in $\SG_m(x)$. If $h \in \SG_m(x)$, this constraint is added as $\max_i \{m(x_i-y_i)+i\}-m = h$, or equivalently as $y_{{\km} (h)}=0$.

Note that any solution of ${\rm MIP}^m(x,h)$ is a numerical semigroup with Frobenius number congruent with $k$ modulo $m$. Since $h \equiv k(h) \pmod m$, $h\in \SG_m(x)$, and the solutions are minimal, if one solution has Frobenius smaller than $h$, then $h$ is not in the set of gaps of those Kunz-coordinates.  Then, this element irrelevant for the decomposition, since there must exist some other semigroup so that $h$ belongs to it.
\end{rem}

Hence, we can simplify further the decomposition process considering only single-objective integer problems rather than multiobjective ones. The following result states this fact.
\begin{theo}
\label{theo:22}
Let $x$ be a Kunz-coordinates vector. Then, the elements in a minimal decomposition of $x$ into $m$-irreducible Kunz-coordinates must belong to the union of the set of optimal solutions of the following problems:
\begin{equation}
\label{ip:x}
\begin{array}{lll}
\min &  \dsum_{i=1}^{m-1} y_i &\\
s.t.&&\\
&y \in {\rm P}_{{\km} (h)}^m(x)\\
&y_{k(h)}=0\\
\end{array}\tag{${\rm IP}^m(x,h)$}
\end{equation}
if $h>2m$ or
\begin{equation}
\label{ipm:x}
\begin{array}{lll}
\min &  \dsum_{i=1}^{m-1} y_i &\\
s.t.&&\\
&y_{k(h)} = x_{k(h)}-2,\\
&y \in {\rm P}^m_m(x)
\end{array}\tag{${\rm IP}_m^m(x,h)$}
\end{equation}
if $h<2m$, for each $h\in \SG_m(x)$.
\end{theo}
\begin{proof}
Problems \eqref{mopoly:x} for any $k=1, \ldots, m-1$ and (${\rm MIP}^m(x,h)$) are multiobjective programs with a full dimension domination cone (see \cite{steuer}). In that case, all the solutions are supported (can be obtained by solving scalars problems, or equivalently, the solutions are in the facets of the convex hull of the integer feasible region) . In our case, when fixing the special gap $h$, we are only interested in one of those nondominated solutions since all of them has $h$ among its gaps, so they are irrelevant for the minimal decomposition.

Furthermore, since the improvement cone, apart of being complete, is the cone generated by $\e_1, \ldots, \e_{m-1}$, and then, the solutions of $\eqref{ip:x}$ and $\eqref{ipm:x}$ are nondominated. Actually, we are looking for solutions with the minimum difference of gaps with $x$, so minimizing $\dsum_i y_i$, and then, it is enough for our purpose to minimize the length of $y$ as formulated in \eqref{ip:x} and \eqref{ipm:x}.
\end{proof}

Note that if \eqref{ipm:x} is feasible, it has a unique feasible solution $y=x-\mathbf{1}-\e_{{\km} (h)}$. Furthermore, this problem is feasible if and only if ${\km} (h) = h-m$ since in that case $h = 2m+  {\km} (h)  - m$, the Frobenius number.

Actually, in this case, if \eqref{ip:x} has also a solution, $y$, it must be also the solution of \eqref{ipm:x}. It is stated in the following theorem.

\begin{theo}
\label{theo:23}
Let $x\in \Z_+^{m-1}$ be a Kunz-coordinates vector, $h \in \SG_m(x)$ and $y^1$ and $y^2$ optimal solutions of problems \eqref{ip:x} and \eqref{ipm:x}, respectively. Then, $y^1=y^2$.
\end{theo}
\begin{proof}
We have two $m$-irreducible undercoordinates of $x$, $x^1=x-y^1$ and $x^2=x-y^2$. $x^1$ is an irreducible Kunz-coordinates vector with Frobenius number $h$. $x^2$ is a Kunz-coordinates vector with Frobenius number $h$ and genus $m$. Since the irreducible Kunz-coordinates are those with maximal genus when fixing the Frobenius number and the maximum genus in this case is $m$, the genus of $x^2$ is also $m$, since in both problems we are minimizing the length of $y$.
\end{proof}

The following result states that when solving \eqref{ip:x}, the optimal value is known.

\begin{lem}
\label{lemma:24}
Let $y$ be an optimal solution of \eqref{ip:x}. Then,
$$
\dsum_i y_i = \dsum_{i=1}^{m-1} x_i - \left\lceil \dfrac{h + 1}{2} \right\rceil
$$
\end{lem}
\begin{proof}
It follows directly form the satisfaction of constraint \eqref{eq:irreducible} and by Lemma \ref{lemma:20}.
\end{proof}

Let $x\in \Z^{m-1}$ be a Kunz-coordinates vector. Once a decomposition is chosen, to select a minimal decomposition, we use a set covering formulation, to choose, among the overall set of minimal $m$-irreducible undercoordinates of $x$, a minimal number of elements for the decomposition.

Let $\SG_m(x)=\{h_1, \ldots, h_s\}$ and $D_i=\{x^{i_1}, \ldots, x^{i_{p_i}}\}$ be the set of the maximal Kunz-coordinates vectors of $m$-irreducible undercoordinates of $x$ when fixing the special gap $h_i$ (optimal solutions of ${\rm IP}^m(x,h_i)$), for $i=1, \ldots, s$. We denote by $D=D_1 \cup \cdots \cup D_s$ the set of $m$-irreducible Kunz-coordinates vectors candidates to be involved in the minimal decomposition of $x$.

We consider the following set of decision variables
$$
z_{ij} = \left\{\begin{array}{rl} 1 & \mbox{ if $x^{ij}$ is selected for the minimal decomposition,}\\
0 & \mbox{ otherwise}.\end{array}\right.
$$
for $i=1, \ldots, m-1$, $j=1, \ldots, i_{p_i}$.

We formulate the problem of selecting a minimal number of $m$-irreducible undercoordinates vectors of $x$ that decompose $x$ into $m$-irreducible Kunz-coordinates as:
\begin{equation}
\label{setcovering}
\begin{array}{lll}
\min &  \dsum_{i=1}^{s} \dsum_{j=1}^{p_i} z_{ij} &\\
s.t.&&\\
& \dsum_{i, j/ mx^{ij}_{{\km} (h)} + {\km} (h) \ge h+1} z_{ij} \ge 1 &, \forall h \in \SG_m(x).
\end{array}\tag{${\rm SC}^m(D)$}
\end{equation}

The covering constraint assures that for each special gap of $x$ there is an element in\\
$\{x^{i1}, \ldots, x^{ip_1}, \ldots, x^{s1}, \ldots, x^{sp_{s}}\}$ such that $h$ is a gap of its corresponding semigroup. Minimizing the overall sum we find the minimum number of Kunz-coordinates fulfilling this requirement. Note that when solving \eqref{setcovering} at most one element in $D_i$ is choosen for each $i=1, \ldots, s$.

Finally, if a numerical semigroup $S$ with multiplicity $m$ is given to be decomposed into $m$-irreducible numerical semigroups, we can, by identifying it with its Kunz-coordinates, give a procedure to compute such a decomposition. This process is described in Algorithm \ref{alg:1}. In that implementation we also consider two trivial cases: (1) when the number of special gaps greater than the multiplicity is $1$, being then the semigroup $m$-irreducible; and (2) when the number of this special gaps is $2$, where the decomposition is given by both solutions of the two unique integer programming problems, and no discarding process is needed.

\begin{algorithm2e}[h]
\label{alg:1}

\SetKwInOut{Input}{Input}
\SetKwInOut{Output}{Output}
\Input{A numerical semigroup $S$ with multiplicity $m$.}

Compute the Kunz-coordinates vector of $S$: $x \in \Z^{m-1}_+$. (Computing the Ap\'ery set.)

${\rm D}=\{\}$.

Compute $\SG_m(x)$.

\eIf{$\#\SG_m(x)=1$}{${\rm DmIR}=\{x\}$}{\For{$h_i \in \SG_m(x)$}{
\eIf{$h_i<2m$}{Set $D:= D \cup \{\mathbf{1}+\e_{{\km} (h)}\}$.}{\For{each optimal solution of $\eqref{ip:x}$, $\hat{y}^i$}{Set $D:= D \cup \{x-\hat{y}^i\}$ .}}}}

Let $D=\{x^{11}, \ldots, x^{1i_1}, \ldots, x^{s1}, \ldots, x^{si_s}\}$.

Let $z^*$ be an optimal solution of \eqref{setcovering}.\\
Set ${\rm DmIR}=\{x^{ij} \in {\rm D}: z_{ij}^*=1\}$

\Output{${\rm DmIRNS} = \{ \langle \{m\} \cup \{mx_i' + i: i =1, \ldots, m-1\} \rangle: x' \in {\rm DmIR}\}$.}
\caption{Decomposition into $m$-irreducible numerical semigroups.}
\end{algorithm2e}

As a consequence of all the above comments and results we state the correctness of our approach.

\begin{theo}
Algorithm \ref{alg:1} computes, exactly, a minimal decomposition into $m$-irreducible Kunz-coordinates vector of a Kunz-coordinates vector $x\in \Z^{m-1}_+$.
Furthermore, the entire set of solutions \eqref{setcovering} characterizes the entire set of minimal decompositions.
\end{theo}

Algorithm \ref{alg:1} is able to compute a minimal decomposition of a Kunz-coordinates vector, $x \in \Z_+^{m-1}$, by enumerating the whole set of optimal solutions of \eqref{ip:x}. However, this task is not easy since, mainly, it consists of enumerating the set of solutions of a diophantine system of inequalities, which is hard to compute (see \cite{gaeryjohnson}). In what follows we propose an approximate approach to obtain a ``short'' decomposition into $m$-irreducibles by choosing an optimal solution of \eqref{ip:x} instead of enumerating all of them. One may choose any of them, but we can also slightly modify the integer programming model to obtain a good solution.

We consider the following set of decision variables:
$$
w_i = \left\{\begin{array}{rl} 1 & \mbox{ if $h_i \in \G(x-y)$,}\\
0 & \mbox{ otherwise}.\end{array}\right.
$$
for $i=1, \ldots, n$, and $\SG_m(x)=\{h_1, \ldots, h_n\}$.

For a fixed $h\in \SG_m(x)$, $w_i=1$ represents that $h_i$ is covered by the solution $x-y$, and then, that can be discarded to obtain a minimal decomposition.

Then, to be sure that we maximize the number of elements in the previous decomposition that can be discarded, we formulate the problem as:

\begin{equation}
\label{ip:xMIN}
\begin{array}{llll}
\max &  \dsum_{i=1}^{\#\SG_m(x)} w_i &\\
s.t.&&\\
&y \in {\rm P}_{{\km} (h)}^m(x)\\
&y_{k(h)}=0\\
&m(x_{{\km} (\hat{h}_i)}- y_{{\km} (\hat{h}_i)})& &\\
&+ {\km} (\hat{h}_i) - \hat{h}_i-1 + M(1-w_i)\ge 0 & \mbox{ forall $\hat{h}_i \in \SG_m(x)$}&\\
\end{array}\tag{${\rm IP}^m_k(x,h)$}
\end{equation}
where $M>>0$

Observe that the {\emph big-$M$} constraint $m(x_{{\km} (\hat{h}_i)}- y_{{\km} (\hat{h}_i)})+ {\km} (\hat{h}_i) - \hat{h}_i-1 + M(1-w_i)\ge 0$ assures that if $\hat{h} \not\in \G(x-y)$ (equivalently $m(x_{{\km} (\hat{h}_i)}- y_{{\km} (\hat{h}_i)})+ {\km} (\hat{h}_i) < \hat{h}_i+1$, then, $w_i=0$. Otherwise, $w_i$ could be $0$ or $1$, but since we are maximizing, $w_i=1$.

The optimal value of this integer problem is then the number of numerical semigroups in the decomposition that can be discarded with this choice.

A pseudocode of the proposed approximated scheme for obtaining a ``short'' decomposition of a Kunz-coordinates vector $x \in\Z^{m-1}_+$ into $m$-irreducible Kunz-coordinates vectors by solving \eqref{ip:xMIN} is shown in Algorithm \ref{alg:2}.

\begin{algorithm2e}[h]
\label{alg:2}

\SetKwInOut{Input}{Input}
\SetKwInOut{Output}{Output}
\Input{A numerical semigroup $S$ with multiplicity $m$.}

Compute the Kunz-coordinates vector of $S$: $x \in \Z^{m-1}_+$. (Computing the Ap\'ery set.)

${\rm D}=\{\}$.

Compute $\SG_m(x)$.

\eIf{$\#\SG_m(x)=1$}{${\rm DmIR}=\{x\}$}{\For{$h_i \in \SG_m(x)$}{
\eIf{$h_i<2m$}{Set $D:= D \cup \{\mathbf{1}+\e_{{\km} (h)}\}$.}{Let $\hat{y}$ be an optimal solution of $\eqref{ip:x}$.  Set $D:= D \cup \{x-\hat{y}\}$ .}}}

Let $D=\{x^1, \ldots, x^s\}$.

\eIf{$\#\SG_m(x)=2$}{${\rm DmIR}={\rm D}$}{Select a minimal decomposition from ${\rm D}$. Let $z^*$ be an optimal solution of \eqref{setcovering}.\\
Set ${\rm DmIR}=\{x^j \in {\rm D}: z_j^*=1\}$}

\Output{${\rm DmIRNS} = \{ \langle \{m\} \cup \{mx_i' + i: i =1, \ldots, m-1\} \rangle: x' \in {\rm DmIR}\}$.}
\caption{Decomposition into $m$-irreducible numerical semigroups.}
\end{algorithm2e}

When running Algorithm \ref{alg:2} we obtain an optimal solution of the problem, and then moving through all the special gaps we obtain a decomposition  into $m$-irreducible Kunz-coordinates. With the following example we show how algorithms \ref{alg:1} and \ref{alg:2} run for a given numerical semigroup.

\begin{ex}
\label{ex:26}
 Let $S=\langle 5, 11, 12, 18 \rangle$. Its Kunz-coordinates vector is $x=(2, 2, 3, 4)$ and  $\SG_m(S) = \{6, 13, 19\}$.

First, we solve one integer problem for each special gap:
\begin{itemize}
\item $h=6$: Since $h < 2\times 5=10$, the integer problem  to solve  is ${\rm P}_5^5(x, 6)$ and then $D_1=\{x^{11}=(2,1,1,1)\}$.
\item $h=13$: In this case $h > 2\times 5=10$ and $h \equiv 3 \pmod 5$, so the  integer problem in this case is ${\rm P}_3^5(x, 13)$. The whole set of optimal solutions is $\{(1, 0, 0, 3), (0,1,0,3)\}$, so $D_2=\{x^{21}=(2,1,3,1), x^{22}=(1,2,3,1)\}$.
\item Finally, for $h=19$: clearly $h > 2\times 5=10$ and $h \equiv 4 \pmod 5$, so the  problem is now  ${\rm P}_4^5(x, 19)$. The set of optimal solutions is $\{(1, 0, 0,0), (0,0,0,1)\}$, and then $D_3=\{x^{31}=(1,2,3,4), x^{32}=(2,2,2,4)\}$.
\end{itemize}
The above five Kunz-coordinates vectors give a decomposition in oversemigroups of $S$. To obtain a minimal decomposition we must solve the associated set covering problem.

Solving ${\rm SC}_5(D)$ we obtain that $z_{11}=z_{31}=1$ and all other variables are set to zero, being then the minimal decomposition given by $x^{11}$ and $x^{31}$, i.e., a minimal decomposition into $5$-irreducible Kunz-coordinates is given by $\{(2,1,1,1), (1,2,3,4)\}$. Translating to numerical semigroups:
$$
S=\langle 5, 11, 7, 8, 9 \rangle \cap \langle 5, 6, 12, 18, 24\rangle
$$

When solving \eqref{ip:xMIN}, we obtain the same decomposition.

\end{ex}

However, the decomposition obtained with Algorithm \ref{alg:2} may be not minimal. The following example illustrates this fact.

\begin{ex}
\label{ex:0}
Let $S=\langle 12, 17, 18, 23, 26, 28, 33, 39 \rangle$ be a numerical semigroup with multiplicity $12$. Its Kunz-coordinates vector is $x=(4,2,3,2,1,1,3,3,2,2,1)$ and $\SG_{12}(x)=\{21,22,27,31,32,37\}$. Then, $6$ integer problems must be solved: ${\rm IP}^{12}_{12}(x,21)$, ${\rm IP}^{12}_{12}(x,22)$, ${\rm IP}^{12}_3(x,27)$, ${\rm IP}^{12}_7(x,31)$, ${\rm IP}^{12}_8(x,32)$ and ${\rm IP}^{12}_1(x,37)$. By solving these problems with {\sf Xpress-Mosel 7.0} \cite{fico} we obtain the following optimal solutions
$x - y \in
\{(1,1,1,1,1,1,1,1,2,1,1)$ $, (1,1,1,1,1,1,1,1,1,2,1),$ $(1,2,3,1,1,1,1,1,1,1,1),$ $(2,2,1,2,1,1,3,1,1,1,1),$ $(1,2,2,2,1,1,2,3,1,1,1),$ $(4,2,1,2,1,1,2,2,1,2,1)\}$

The translations of the above coordinates in terms of numerical semigroups are \\
$
\{\langle 12,13,14,15,16,17,18,19,20,22,23,33\rangle ,$
$\langle 12, 13, 14, 15, 16, 17, 18, 19, 20, 21, 34, 23\rangle ,$\\
$\langle 12,13,16,17,18,19,20,21,22,23\rangle ,$
$\langle 12,15,17,18,20,21,22,23,25,26,28,43\rangle ,$\\
$\langle 12,13,17,18,21,22,23,26,27,28,31,44\rangle ,$
$\langle 12,15,17,18,21,23,26,28,31,32,34,49\rangle \}
$

Now, by solving problem \eqref{setcovering}, $\langle 12, 13, 14, 15, 16, 17, 18, 19, 20, 21, 34, 23\rangle$ is discarded. Then, the decomposition using our methodology is given by five $12$-irreducible numerical semigroups:\\
$S= \langle 12,13,14,15,16,17,18,19,20,22,23,33\rangle \cap
\langle 12,13,16,17,18,19,20,21,22,23\rangle \cap$\\
$\langle 12,15,17,18,20,21,22,23,25,26,28,43\rangle \cap
\langle 12,13,17,18,21,22,23,26,27,28,31,44\rangle \cap$\\
$\langle 12,15,17,18,21,23,26,28,32,34\rangle
$

However, this decomposition is not minimal since
$
S = \langle 12,13,16,17,18,19,20,21,22,23,26,39\rangle \cap
\langle 12,15,17,18,20,21,22,23,25,26,28,43\rangle \cap
\langle 12,13,17,18,21,22,23,26,27,28,31,44\rangle \cap$\\
$\langle 12,15,16,17,18,23,26,31,32,33,34,49\rangle
$
is a decomposition into $m$-irreducible numerical semigroups using a smaller number of terms.
\end{ex}

The situation of Example \ref{ex:0} is due to the fact that among the whole set of optimal solutions of \eqref{ip:x}, Algorithm \ref{alg:2} chooses a particular one, but depending of that choice more or less elements can be discarded from that decomposition to obtain the minimal one. To avoid this fact, we need to consider a compact model that connects all the possible elements in the decomposition and that selects, among all of  them, the smallest number of solutions to decompose a Kunz-coordinates vector.

\section{A compact model for minimally decomposing into $m$-irreducible Kunz-coordinates vectors}
\label{sec:5}

In the section above we describe an exact and a heuristic procedure to compute a minimal decomposition of a Kunz-coordinates vector $x\in \Z^{m-1}$ into $m$-irreducible Kunz-coordinates. To obtain solutions by using that exact procedure we need to enumerate the solutions of a knapsack type diophantine equation included in the Kunz polytope. Once we have those solutions, a set covering problem must be solved to obtain a minimal decomposition. By using that model, the complete enumeration cannot be avoided since by choosing one solution, one may obtain non-minimal decompositions when solving the set covering model (see Example \ref{ex:0}).  We present here a compact model to decompose any Kunz-coordinates vector, $x \in \Z_+^{m-1}$,  merging in a single integer programming problem all the subproblems considered in the previous section to ensure minimal decompositions. Moreover, this approach will allow us to prove a polynomiality result for the problem of decomposing into $m$-irreducible numerical semigoups.

Let $\SG_m(x)=\{h_1, \ldots, h_s\}$.

We consider the following families of decision variables for the new model:

\begin{itemize}
\item $y_i^l \in \Z_+$, such that $x-y^l$ is a $m$-irreducible undercoordinate of $x$ with Frobenius number $h_l$, for all $l=1, \ldots, s$ and $i=1, \ldots, m-1$.
\item $w_l \in \{0,1\}$, representing if $x-y^{h_l}$ is chosen $(1)$ or not $(0)$ for a minimal decomposition into $m$-irreducible coordinates of $x$, for all $l=1, \ldots, s$.
\item $z_k^l \in \{0,1\}$, that measures if $h_k$ is a gap of $x-y^l$ ($1$) or not $(0)$, for all $l, k=1, \ldots, s$. Note that $h_k \in \G(x-y^l)$ if and only if $y^l_{k(h_k)}=0$.
\end{itemize}

Then, the proposed model, ${\rm CIP}^m(x)$, is described as follows:
\begin{equation}
\label{cm}\tag{${\rm CIP}^m(x)$}
\min \;\; \dsum_{l=1}^{s} w_l
\end{equation}
s.t.
\begin{align}
& y_i^l \leq x_i-1 & \mbox{ forall $i=1, \ldots, m-1$ forall $l=1, \ldots, s$,}\label{r4}\\
& y_i^l + y_j^l - y_{i+j}^l \leq x_i + x_j - x_{i+j} & \mbox{ if $i+j<m$, forall $l=1, \ldots, s$,}\label{r5}\\
& y_i^l + y_j^l - y_{i+j-m}^l \leq x_i + x_j - x_{i+j-m} +1& \mbox{ if $i+j>m$, forall $l=1, \ldots, s$,}\label{r6}\\
& \dsum_{i=1}^{m-1} y_i^l = (\dsum_{i=1}^{m-1} x_i - \left\lceil\dfrac{h_l+1}{2}\right\rceil) w_l & \mbox{forall $l=1, \ldots, s$ with $h_l>2m$,}\label{r7}\\
& \dsum_{i=1}^{m-1} y_i^l = (\dsum_{i=1}^{m-1} x_i - m) w_l & \mbox{forall $l=1, \ldots, s$ with $h_l<2m$,}\label{r8}\\
& y^l_{k(h_l)}=0 & \mbox{forall $l=1, \ldots, s$,}\label{r9}\\
& \dsum_{l} z^l_{k(h_k)} \geq 1 & \mbox{forall $k=1, \ldots, s$,}\label{r10}
\end{align}
\begin{align}
 & z^l_{{\km} (h_k)} \geq 1- y^l_{{\km} (h_k)} - M\,(1-w_l) &\mbox{forall $l, k=1, \ldots, s$},\label{r11}\\
& y^l_{k(h_k)} \leq M(1-z^l_{k(h_k)}) & \mbox{ forall $k=1, \ldots, s$,}\label{r12}\\
& z_{k(h_k)}^l \leq w_l &\mbox{forall $l, k=1, \ldots, s$,}\label{r13}\\
& y_i^l \in \Z_+, &\mbox{ forall $i=1, \ldots, m-1$, $l=1, \ldots, s$,}\label{r14}\\
& w_l \in \{0,1\}, &\mbox{ forall $l=1, \ldots, s$,}\label{r15}\\
& z_j^l \in \{0,1\} &\mbox{ forall $l =1, \ldots, s$, $j=1, \ldots, m-1$.}\label{r16}
\end{align}

for $M\geq \max\{x_{k(h_l)}: l=1, \ldots, s\}$.

The components of any optimal solutions, $y^*$, of the above problem in the set $\{y^{*l}: y^{*l} \neq 0, l=1, \ldots, s\} = \{y^{*l_1}, \ldots, y^{*l_p}\}$ give a minimal decomposition of $x$ into $m$-irreducible Kunz-coordinates vectors as $\{x-y^{*l_j}: j=1, \ldots, s\}$. Note also that $\F(x-y^{*l_j})= h_{l_i}$.

Constraints \eqref{r4}-\eqref{r6} assure that $x-y^l$ is a undercoordinate of $x$. \eqref{r7} and \eqref{r8} give conditions related to the genus and the Frobenius number of those Kunz-coordinates vectors (Corollary \ref{cor:5}) associated to the choice of $y^l$ ($w^l=1$). Constraint \eqref{r9} assures that $h_l$ is a gap of $x-y^l$, and \eqref{r10} that there is at least one element in the decomposition having $h_l$ among its gaps. Constraints \eqref{r11}-\eqref{r13} control that the variables $z_k^l$ are well-defined. \eqref{r14}-\eqref{r16} are the integrality and binary constraints for the variables.

The optimal value of \eqref{cm} gives the number of Kunz-coordinates involved in a minimal decomposition of $x$ into $m$-irreducible Kunz-coordinates vectors.

The solution of \eqref{cm} gives exactly a minimal decomposition of $x$ into $m$-irreducible Kunz-coordinates
(or $m$-irreducible numerical semigroups). However, it is  harder to solve than the problems in Algorithm \ref{alg:2} since it has much
more variables (by using Algorithm \ref{alg:2}, we need to solve at most $m-1$ problems with $m-1$ variables and a set covering problem with at most $m-1$ variables while \eqref{cm} has $2(m-1)^2+(m-1)$ integer/binary variables). In the computational experiments (see Section \ref{sec:5}) we have noticed that the solutions when running Algorithm \ref{alg:2} are not far from minimality and it is faster than running Algorithm \ref{alg:1} or solving \eqref{cm}.

\begin{rem}[$m$-symmetry and $m$-pseudosymmetry]
In \cite{blanco-rosales10} it is also defined the notion of $m$-symmetry and $m$-pseudosymmetry of a  numerical semigroup with multiplicity $m$, extending the previous notions of symmetry and pseudosymmetry (see \cite{Ro-GS09}). A numerical semigroup, $S$, with multiplicity $m$ is $m$\textit{-symmetric} if $S$ is $m$-irreducible and $\F(S)$ is odd. On the other hand, $S$ is $m$\textit{-pseudosymmetric} if $S$ is $m$-irreducible and $\F(S)$ is even.

Rosales and Branco analyzed in \cite{rosales02} and \cite{rosales02b} those numerical semigroups
that can be decomposed into symmetric numerical semigroups (in this case the semigroup is called ISY-semigroup).
Another interesting application of this methodology is to  compute a decomposition of $S$ into $m$-symmetric numerical
 semigroups (following the notation in \cite{rosales02b}, $S$ is an ISYM-semigroup). This follows by fixing in \eqref{cm}
 that the $m$-irreducible numerical oversemigroups of $S$ associated to even special gaps do not appear in the decomposition
($y^l_i=0$ for all $i=1, \ldots, m-1$ if $l$ is even). Thus, the $m$-irreducible numerical semigroups which Frobenius numbers
are each one of the odd special gaps must cover the whole set of gaps. If this problem is feasible, its solution gives a minimal
decomposition into $m$-symmetric numerical semigroups. However, in this case we cannot ensured that it is always possible to
decompose into $m$-symmetric numerical semigroups (for instance, a numerical semigroup with even Frobenius number is not decomposable
in this way). Then, if problem \eqref{cm} is infeasible, the semigroup cannot be expressed as an intersection of $m$-symmetric numerical semigroups.

In addition, \cite{rosales02b} analyzes the set of ISYG-semigroups (those that can be expressed as an intersection of symmetric semigroups with the same Frobenius number). We could introduce the notion of ISYGM-semigroups (those that can be expressed as an intersection of symmetric numerical semigroups with the same Frobenius number and multiplicity). This case can be also handled with our approach by fixing the Frobenius number of the semigroup in \eqref{cm}.

 An similar methodology can be applied to compute a decomposition into $m$-pseudosymmetric numerical semigroups.
\end{rem}

\begin{rem}[Computational Complexity]
Assume that $m$ is fixed. \eqref{cm} has at most $2(m-1)^2+(m-1)$ variables and then, it is solvable in polynomial time \cite{lenstra}.
It is also worth noting that the heuristic approach also has polynomial time overall complexity. Indeed, for each special gap of $x$, one integer program is solved, ${\rm IP}^m (x,h)$ if $h>2m$ or ${\rm IP}^m_{m}(x)$ if $h<2m$. Since the number of special gaps is bounded above by $m-1$, the complexity of this step is polynomial for fixed multiplicity, so polynomial. Once we have the solutions for all the special gaps, the discarding step consists of solving the set covering problem \eqref{setcovering} with at most $m-1$ variables, so polynomial in $m$.

On the other hand, , the algorithm proposed in \cite{blanco-rosales10} to decompose a numerical semigroup $S$ with multiplicity $m$ into $m$-irreducible numerical semigroups can be rewritten as follows.

Let $\mathcal{G}_x= (V,E)$ be a directed graph whose set of vertices is the set of undercoordinates of $x$, $\mathcal{U}_m(x)$, and $(x^1, x^2) \in E$ if $x^2=x^1 - \e_{h\pmod m}$ for some $h \in \SG_m(x)$. Figure \ref{fig:graph2} illustrates how this graph is built. In that figure we denote $\SG_m(x)=\{h_1, \ldots, h_k\}$ and $\SG_m(x + \e_{{\km} (h)}=\{h_1', \ldots, h'_{k}\}$. The algorithm looks for a set o vertices $\{x^1, \ldots, x^n\}$ with the properties that $\#\SG_m(x^i)=1$ for all $i=1, \ldots, n$ and that any other vertex is dominated by any of the elements in the set. Furthermore, $\mathcal{G}_x$ is a tree since it does not have circuits. In \cite{blanco-rosales10}, a breath first search over this tree is proposed to find the desired set. Clearly, the worst case complexity of this method is exponential even for fixed multiplicity.

\begin{figure}[h]
{\tiny$$
\xymatrix{ & &   x\ar[dl]\ar[dr] &\\
& x - \e_{{\km} (h_1)}\ar[dl]  \ar[dr] &  \cdots &  x - \e_{{\km} (h_k)}\\
x - \e_{{\km} (h_1)} - \e_{{\km} (h'_1)} &  \cdots &x - \e_{{\km} (h_1)} - \e_{{\km} (h'_k)} &&}
$$}
\caption{Sketch of $\mathcal{G}_x$\label{fig:graph2}.}
\end{figure}
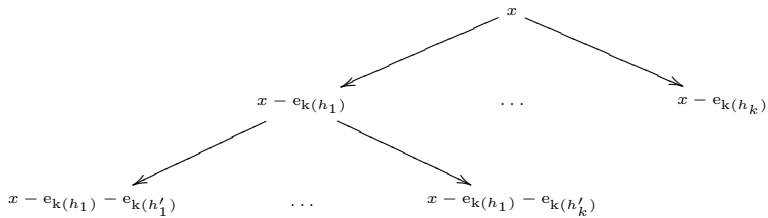
\end{rem}

\section{Computational Experiments}
\label{sec:6}

In this section we present the results of some computational experiments performed to analyze the applicability of the proposed algorithm. Our algorithm has been implemented in {\sf XPRESS-Mosel 7.0} \cite{fico} that allows to solve the single-objective integer problems involved in the decomposition into $m$-irreducible numerical semigroups, by using a branch-and-bound method and nesting models by calling the library \texttt{mmjobs}. The algorithms have been executed on a PC with an Intel Core 2 Quad processor at 2x 2.50 Ghz and 4 GB of RAM.

The complexity of the algorithm depends of the dimension of the space (multiplicity), the size of the coefficients of the constraints and the number of special gaps. Then, we have randomly generated three different batteries of numerical semigroups with the following requirements:
\begin{description}
\item[Battery I] Numerical semigroups with multiplicities ranging in $[0,25]$ (divided in the five subintervals $(0,5]$, $(5,10]$, $(10, 15]$, $(15,20]$ and $(20, 25]$) with generators ranging in $[2, 5000]$. Ten instances for each subinterval.
\item[Battery II] Numerical semigroups with multiplicities ranging in $[10,2000]$ (divided in the seven subintervals $(10, 25]$, $(25, 50]$, $(50, 100]$, $(100, 250]$, $(250, 500]$, $(500, 1000]$, $(1000, 2000]$) with generators ranging in $[2, 5000]$. Five instances for each subinterval.
\item[Battery III] Numerical semigroups with multiplicities ranging in $[25,150]$ (divided in the five subintervals $(25, 50], (50, 75], (75, 100], (100,
125]$ and $(125, 150]$) with generators ranging in $[2, 5000]$ and with number of special gaps greater than the multiplicity less than or equal to $30$. Ten instances for each subinterval.
\end{description}
The first battery of problem is designed to compare the three algorithms: the one implemented in
 {\sf GAP}, the heuristic approach (Algorithm \ref{alg:2}) and the compact model \eqref{cm}. With the second set of problems,
we check the efficiency of Algorithm \ref{alg:2} for solving large instances. Finally, with the third test set,
we compare the difficulty of solving \eqref{cm} comparing to the heuristic algorithm. (Note that this difficulty is mainly due to the number of special gaps since it increases the number of variables.) Therefore, we generate numerical semigroups
with very large multiplicities but where the number of special gaps is bounded above by $30$.

We have used recursively the function \texttt{RandomListForNS} of {\sf GAP}\cite{numericalsgps} until we
found the list of integers defining the semigroup  with the above requirements. The implementation done for decomposing in {\sf GAP}
into $m$-irreducible numerical semigroups is an adaptation of the function \texttt{DecomposeIntoIrreducibles}
for decomposing into standard irreducible numerical semigroups.

The results of these experiments are summarized in tables \ref{table:1}--\ref{table:3}. In these tables,
$\texttt{m}$ indicates the range of the multiplicity, $\texttt{CMtime}$ and $\texttt{Heurtime}$ the average times in seconds consumed
 by solving \eqref{cm} and Algorithm \ref{alg:2}, respectively, in {\sf Xpress-Mosel}, \texttt{GAPtime} informs on the average time
consumed by {\sf GAP} for the same task, $\#\SG$ is the average number of special gaps of the problems and
 $\#$\texttt{m-irred} is the average number of semigroups involved in a minimal decomposition. The column
 \texttt{avgap} is the average  difference between  the number of numerical semigroups used in the heuristic decomposition and the number of numerical semigroups used in the minimal decomposition computed by solving \eqref{cm}.

Note that {\sf GAP} was not able to solve any of the 10 instances when the multiplicity ranges in $(20, 25]$.

\begin{table}[h]
\begin{center}
    \begin{tabular}{|c|ccc|cc|c|}
  \hline
  $\texttt{m}$  & $\texttt{CMtime}$ & $\texttt{Heurtime}$  & $\texttt{GAPtime}$ &$\# \SG$ & $\#\texttt{m-irred}$ & \texttt{avgap}\\\hline
    [0,5] & 0.001 & 0.020 & 0.001 & 1.5   & 1.5 & 0\\
    (5,10] & 0.003 &0.054 & 2.3973 & 2.7   & 2.3 & 0 \\
    (10, 15] & 0.013 &0.091 & 4.1645 & 4.1   & 3.4 & 0.1\\
    (15,20] & 0.053 &0.081 & 523.556 & 5.4   & 4 & 0\\
    (20,25] & 0.046 & 0.089 & n/a & 5.7   & 4.4 & 0.1\\
        \hline
    \end{tabular}
    \vspace*{0.25cm}
\caption{Results of the computational experiments for Battery I.\label{table:1}}
\end{center}
\end{table}

\begin{table}[h]
\begin{center}
    \begin{tabular}{|c|c|cc|}
  \hline
    $\texttt{m}$  & $\texttt{Heurtime}$  & $\# \SG$ & $\#\texttt{m-irred}$\\\hline
     (25,50]  & 0.242 & 11.8  & 7.2 \\
     (50,100]  & 1.411 & 19.6  & 9.6 \\
     (100,250]  & 168.272 & 42.4  & 25.4 \\
     (250,500]  & 1318.475 & 86.2  & 47.8  \\
     (500,1000]  &  1056.878     &  27.2     & 18.8 \\
     (1000,2000]  &  1895.058     &   15.2    & 9.8 \\\hline
    \end{tabular}
        \vspace*{0.25cm}
\caption{Results of the computational experiments for Battery II.\label{table:2}}
\end{center}
\end{table}

\begin{table}[h]
\begin{center}
    \begin{tabular}{|c|cc|cc|c|}
  \hline
  $\texttt{m}$  & $\texttt{CMtime}$ & $\texttt{Heurtime}$  &$\# \SG$ & $\#\texttt{m-irred}$ & \texttt{avgap}\\\hline
     (25,50]   & 1.064   & 0.201  &  9.3 & 5.8  &  0.7  \\
     (50,75]   & 6.981   & 0.713  &   13.5&  7.1 &   1.1 \\
    (75,100]   & 58.580   & 1.819  &   16.3&  9 &   1 \\
    (100,125]  & 102.999   & 3.428  &   15.1&  7.1 &  1.6  \\
    (125,150]  &  144.531  & 5.752  &  15.5 &  8.3 &  1.3  \\\hline
    \end{tabular}
        \vspace*{0.25cm}
\caption{Results of the computational experiments for Battery III.\label{table:3}}
\end{center}
\end{table}

We have also observed that the algorithm implemented in {\sf GAP} does not ensure minimal decompositions into $m$-irreducible numerical semigroups.
For instance, consider the following case: $S=\langle  15, 17, 19, 48, 52, 59, 73 \rangle$ that decomposes in {\sf GAP} into six $15$-irreducible
numerical semigroups while our methodology obtains a decomposition into five $15$-irreducible numerical semigroups.  The reason why {\sf GAP} fails is
closely related to the same fact that prevents to ensure, in all cases,  Algorithm \ref{alg:2} to get minimal solutions

From our computational experiments we observe that except for the instances with $m \in [0,5]$, where the algorithm in {\sf GAP} spends almost
the same time to compute the decompositions, our methodology solves the problems faster than {\sf GAP}.
Actually, in this battery solving the problem \eqref{cm} is the best way to compute such a decomposition. This is due to the minimum
computational time consumed by {\sf Xpress-Mosel} to load the problems  involved in Algorithm \ref{alg:2}.

Both, the exact algorithm based on solving \eqref{cm} and the heuristic approach are able
to compute, in reasonable CPU times, minimal decompositions into $m$-irreducible numerical semigroups
for multiplicities up to $150$ while the procedure implemented in {\sf GAP} is not able to solve problems
 with multiplicities ranging even in $(20, 25]$. Furthermore, although the default branch-and-bound algorithm
is not able to solve \eqref{cm} for larger multiplicities, the heuristic approach solves problems with multiplicities up  to $m=2000$.

The heuristic approach finds, much faster than the exact approach, a short decomposition of a numerical semigroups into $m$-irreducible numerical semigroups. Furthermore, the heuristic approach reaches most of the times a minimal decomposition. For instance, in the first battery of problems, the heuristic value does not coincide with the exact optimal one in only two out the fifty instances. Moreover,  the third battery of instances
satisfies that in 30\% of the cases the minimal decomposition coincides with the heuristic short decomposition, in 34\% of the cases the difference is only one semigroup, in 30\% of the cases is two semigroups, in 4\% (two cases) is three and in only 2\% (one instance) is four.

Note that most of the computations done by using Algorithm \ref{alg:2} may be parallelized by solving in different cores each one of the problems \eqref{ip:xMIN} since they are independent. This could improve the CPU times and sizes of the problems because more than $99\%$ of the time  consumed by this algorithm is to solving those problems, while just a little part of the time is spent solving the set covering problem.

On the other hand, we have only implemented the proposed models in {\sf Xpress-Mosel}, with the default branch-and-bound method. Larger instances could be solved by applying specific more sophisticated integer programming algorithms to solve each one of the problems.

\section{Acknowledgments}
This research has been partially supported by the Spanish Ministry of Science and Education grants {\it MTM2007-
67433-C02-01} and {\it MTM2010-19576-C02-01}. The first author have been also supported by Juan de la Cierva grant {\it JCI-2009-03896}.

\end{document}